\documentclass[12pt]{amsart}
\usepackage[utf8]{inputenc}
\usepackage[english]{babel}
\usepackage{amsmath,amssymb,amsfonts,amsthm, soul, xcolor}
\usepackage{graphicx}
\usepackage{tikz}

 \usepackage[draft]{fixme}
\newcounter{thm}
\newtheorem{theorem}[thm]{Theorem}
\newtheorem*{theorem*}{Theorem}
\newtheorem{lemma}[thm]{Lemma}

\newtheorem{proposition}[thm]{Proposition}
\newtheorem*{conjecture*}{Conjecture}

\theoremstyle{definition}
\newtheorem{definition}[thm]{Definition}

\renewcommand{\Im}{\mathop{\mathrm{Im}}}
\newcommand{\dist}{\operatorname{dist}}
\newcommand{\eps}{\varepsilon}
\newcommand{\bbC}{\mathbb C}
\newcommand{\CC}{\mathbb C}
\newcommand{\DD}{\mathbb D}
\newcommand{\NN}{\mathbb N}
\newcommand{\QQ}{\mathbb Q}
\newcommand{\RR}{\mathbb R}
\newcommand{\TT}{\mathbb T}
\newcommand{\ZZ}{\mathbb Z}
\newcommand{\bbR}{\mathbb R}
\newcommand{\bbQ}{\mathbb Q}

\newcommand{\bbN}{\mathbb N}

\newcommand{\cB}{\mathcal B\mathcal T}
\newcommand{\cD}{\mathcal D}
\newcommand{\cR}{\mathcal R}

\newcommand{\cW}{\mathcal W}

\newcommand{\cC}{\mathcal C}
\newcommand{\cAC}{\mathcal A\mathcal C}

\renewcommand{\dist}{\mathrm{dist}\,}

\renewcommand{\mod}{\operatorname{mod}}

\newcommand{\ignore}[1]{}

\numberwithin{equation}{section}
\numberwithin{thm}{section}
\author{Nataliya Goncharuk, Michael Yampolsky}
\title{Rotation domains for maps of bounded type}

\begin{document}
\begin{abstract}
We present a novel approach for deriving KAM-type linearization theorems directly -- and almost immediately -- from the existence of the stable foliation for a renormalization operator. We give a few illustrations in dynamics in one and several complex variables, starting with a version of the classical theorem of Arnol'd and ending with a result on persistence of Herman rings in families of two-dimensional maps.
  
  \end{abstract}

\maketitle

\section*{Foreword}
The purpose of this note is to describe a novel approach to KAM-type linearization theorems, such as Risler's Theorem \cite{Risler,GY} or the second author's result on the persistence of Herman rings in 2D \cite{KAM-yam}. The proofs here rely only on the existence of a stable foliation for a complexified renormalization operator. The connection of renormalization with linearization of rotational dynamics in this context is well-established (see, for instance, \cite{KhSin87,Yoccoz2002} and references therein). However, in those works, renormalization is used as a geometric tool for the study of dynamics. We consider it as a dynamical system in its own right (cf. our earlier work \cite{GY}), whose hyperbolic structure leads to the existence and regularity of analytic stable manifolds in the parameter space of dynamical systems. Our derivation of the existence of linearization is a direct and almost immediate consequence. We restrict our attention to rotation numbers of bounded type, to ensure uniformity of hyperbolicity of the action of renormalization. We have aimed to keep the presentation simple and brief to highlight the main ideas, which should find other applications in KAM-type setting. In view of this, we skip some of the technical details which have been covered in our previous works and reference those works instead whenever possible. 

\section{Introduction}

We denote $\{x\}$, $[x]$ the fractional and the integer parts of a real number $x$ respectively. We let $\TT$ denote the circle $\RR/\ZZ$, and for $\alpha\in\RR$ we let $R_\alpha$ be the rotation of the circle
$$R_\alpha(x)=x+\alpha/\ZZ.$$
We identify the points in $\TT$ with their representatives in $[0,1)$ where convenient. For $\eps>0$ we set
  $$\Pi_\eps=\{|\Im z|<\eps\}/\ZZ\subset \CC/\ZZ.$$
  
For $\alpha\in(0,1)$, let $$G(\alpha)=\left\{\frac{1}{\alpha} \right\}$$
denote the Gauss map. We set
$$\alpha_0\equiv\alpha,\ldots,\alpha_n=G(\alpha_{n-1}),\ldots;$$
this sequence is infinite if and only if $\alpha\notin \QQ$, otherwise, we will end it at the last non-zero term.
Note that numbers $a_n=[1/\alpha_{n}]$, $n\geq 0$ are the coefficients of a continued fraction expansion of $\alpha$ with positive terms, which is unique  if $\alpha\notin \bbQ$:
$$\alpha=\cfrac{1}{a_0+\cfrac{1}{a_1+\cfrac{1}{a_2+\cdots}}},$$
which, to save space, we abbreviate as 
$$\alpha=[a_0,a_1,a_2,\ldots].$$
As usual, $p_n/q_n$ will denote the $n$-th convergent of the continued fraction of $\alpha$:
$$\frac{p_n}{q_n}=[a_0,\ldots,a_{n-1}].$$
For $K\in\NN$, we say that $\alpha$ {\it is of a type bounded by }$K$ iff
$$\sup a_i\leq K,$$
and denote $\cB_K$ the set of all such irrationals in $(0,1)$. Classically, the union $$\cB=\cup_{K\in\NN}\cB_K$$ of all numbers of a bounded type coincides with  the set of Diophantine numbers with exponent $2$.

Furthermore, denote $\cD_\eps$ the affine Banach space of bounded analytic maps $\Pi_\eps\to\CC/\ZZ$ which continuously extend to the boundary with the sup-norm.

The goal of this note is to give  novel and ``simple'' proofs of analytic linearization theorems for rotation numbers of bounded type. Namely, we prove:

\begin{theorem}
  \label{th-main-1}
 Let $\eps>0$. For $K\in\NN$ there exists $\kappa=\kappa(K,\eps)>0$ such that the following holds.
Suppose $f\in\cD_\eps$ is an analytic diffeomorphism of the circle $\TT$ with rotation number $\alpha=\rho(f)\in\cB_K$, which is $\kappa$-close to $R_\alpha$ in $\cD_\eps$. Then $f$ is conformally conjugate to $R_\alpha$ in $\Pi_{\eps/2}$.
  \end{theorem}

\begin{theorem}
  \label{th-main-2}
  Let $\alpha\in \cB_K$ and $\eps>0$. Then there exists $\kappa=\kappa(K,\eps)>0$ such that the following holds.
The set of maps $f\in \mathcal D_{\eps}$ such that $f$ is analytically conjugate to $R_\alpha$ in  $\Pi_{\eps/2}$ and whose distance  to $R_\alpha$ is bounded by $\kappa$ forms an embedded analytic submanifold of $\mathcal D_\eps$ at $R_\alpha$ of codimension 1.
\end{theorem}

Finally, let $$\Delta_{\eps}\equiv \Pi_\eps\times \DD_\eps(0)\subset\CC/\ZZ\times\CC$$
and denote $\cW_{\eps}$ the Banach space of analytic maps $\Delta_{\eps}\to\CC/\ZZ\times\CC$ which are continuous up to the boundary, with the sup-norm. 
\begin{theorem}
  \label{th-main-3}
  Let $\alpha\in \cB_K$ and $\eps>0$.
Consider a generic family\footnote{This theorem is stated differently from the previous two, since renormalization of two-dimensional maps is defined not in terms of maps but of almost commuting pairs (see \S~\ref{sec:2d} for the definitions). Genericity of the family  is understood in the sense of the pairs produced by the maps: the corresponding family of pairs should be transversal to the stable manifold of the renormalization operator.} of maps $F_\lambda\in \cW_\eps$ with analytic dependence on $\lambda\in\DD^n$ such that $$F_0(z,w)=(R_\alpha(z),0).$$
  Then there exists an $(n-1)$-dimensional submanifold $W\subset \DD^n$ with $0\in S$ such that for all
  $\lambda\in W$, the map $F_\lambda$ has a biholomorphically embedded  complex 1-dimensional annulus $A_{F_\lambda}$ on which it is analytically conjugated to $R_\alpha$ (a Herman ring).
\end{theorem}

Let us comment on the statements first. Theorem~\ref{th-main-1} is, of course, just a particular case of the very classical theorem of Arnol'd \cite{Arnold61}, which posited the corresponding statement for all Diophantine rotation numbers. However,  Theorem~\ref{th-main-2} is already highly non-trivial. It is a particular case of a version of Risler's theorem \cite{Risler} proven by the authors in \cite{GY}. The proof in \cite{GY} is for all Brjuno rotation numbers -- but it relies on Yoccoz's famous result \cite{Yoccoz2002} which extends Arnol'd's theorem to the Brjuno class of rotation numbers. Theorem~\ref{th-main-3} is the main result of a recent paper \cite{KAM-yam} of the second author.

Our proofs of all three theorems are based on the same scheme. Firstly, we construct an analytic renormalization operator which acts in the corresponding Banach space and has $$\Lambda_K\equiv \{R_\alpha\;|\;\alpha\in\cB_K\}$$ as a uniformly  hyperbolic compact invariant set with one unstable direction. We then use a simple trick to show that a map lying in the local stable manifold of $R_\alpha$ is analytically conjugate to $R_\alpha$ in an annulus. The latter part of the proof applies universally whenever it can be shown that $R_\alpha$ has an analytic stable submanifold of codimension $1$. Of course, in our work \cite{GY}, we have established this fact for all $\alpha$ in the Brjuno class -- but that proof requires  Yoccoz's theorem as an extra ingredient. This one does not.

For the remainder of the paper, let us fix $K\in\NN$. Note that $\Lambda_K$ is a Cantor set, since $\cB_K$ is invariant under the action of finitely many branches of the expanding Gauss map $G(x)=\{1/x\}$.
We begin with the one-dimensional case, Theorems~\ref{th-main-1} and \ref{th-main-2}.

\section{Renormalization of maps in $\cD_\eps$, following \cite{GY}}

\subsection{Definition of $\cR$}
Fix  $\eps>0$.
Recall that in \cite{GY}, maps close to a rotation $R_\alpha$ are renormalized using the following {\it cylinder renormalization} procedure.

Suppose that $f\in \mathcal D_\eps$ is close to a  rotation $R_\alpha$, $\alpha\in \cB_K$ (in particular, $f$ is univalent).
Consider a vertical segment $I \subset \Pi_{\eps}$ that contains zero. For an integer $n$,  let $V$ be the curvilinear rectangle bounded by $I$, $f^{q_n}(I)$, and two straight segments joining their endpoints. If  $f\in \mathcal D_{\eps}$ is sufficiently close to a rotation, these four curves are simple and bound a domain in $\Pi_{\eps}$. We glue neighborhoods of $I$ and $f^{q_n}(I)$ by $f$ and obtain a Riemann surface whose topological type is that of a cylinder. Then uniformize it using a bi-antiholomorphic map $\Psi_f$. This is done in such a way that $0\mapsto 0$,  the dependence $f\mapsto\Psi_f$ is analytic, and $\Psi_f$ is real-symmetric for a real-symmetric $f$.

The first-return map to $V$ projects to a locally conformal map $g\equiv \cR _n f$ on a neighborhood 
of $\TT$.
If $f=R_\alpha,$ then $\Psi_f$ is the antilinear rescaling by $\{n\alpha\}$ and $\cR_n f=R_{G(\alpha)}.$
Note that $\cR_n f$ depends on $f$  locally analytically. Indeed, in a small neighborhood of $R_\alpha\in\Lambda_K$ it is the rescaling of the iterate $f^{q_n}$ by $\Psi_f$. 
Finally, 
if $f$ is close to a rotation, then $I$ can be chosen close to the vertical segment $[-i\eps,i\eps]$ spanning the whole height of the cylinder. To have a better control on distortion of rescaling $\Psi_f$, we  select a smaller $I$: namely, we take $\eps'>\eps$, let $l=|f^{q_n}(0)|$ be the length of the fundamental interval, and choose $I = [-i\eps' l, i\eps' l]$. The length $l$ can be made arbitrarily small by increasing $n$, so we can guarantee $\eps'l<\eps$.

By compactness of $\Lambda_K$, we can choose $\delta>0$, and an even value of $n\in\NN$ such that if $f$ is $\delta$-close to $\Lambda_K$, then  $\cR_n f$ is defined in $ \Pi_{\eps'}$ with $\eps'>\eps$.
We fix this choice and obtain \cite{GY}:

\begin{theorem}
  \label{th-GY-1}
With the above choices, $\cR_n$ is a compact, analytic, real-symmetric operator from the $\delta$-neighborhood of $\Lambda_K$ in $\cD_\eps$ to $\cD_\eps$. The set $\Lambda_K$ is invariant.
  \end{theorem}

\subsection{Hyperbolicity of renormalization}
We further have \cite{GY}:

\begin{theorem}
  \label{th-GY-2}
For sufficiently large $n=n(K,\eps)$, the invariant compact  set $\Lambda_K$ is uniformly hyperbolic for $\cR_n$ with 
one-dimensional unstable direction.
  \end{theorem}

It is evident that  $\{R_\zeta, \zeta \in \bbC\}$ is an unstable direction, on which $\cR_n$ acts as an iterate of the Gauss map.

The construction of the stable bundle for $D\cR_n|_{\Lambda_K}$ requires more cleverness. In \cite{GY} it is shown to coincide with the (complex) tangent sub-bundle $V_0$ consisting of vector fields with zero average on the circle:
$$\displaystyle \int_\TT vdz=0.$$
The strategy of the proof is as follows. Firstly, it is shown
in \cite{GY} that, in full generality, analytic conjugacy classes of irrational rotations lie inside their stable manifolds. Specifically:
 \begin{proposition}\label{prop-conj-class}
   Let $\eps>0$ and $\alpha\in \cB_K$. There exists $\lambda\in(0,1)$  such that the following holds. Suppose $f\in\cD_\eps$ is locally conformally conjugate to $R_\alpha$ {via $\xi\colon\Pi_{\eps/2}\to \Pi_\eps$.\, } Then for sufficiently large $n$, for all $k$, the cylinder renormalizations $(\cR_n)^k(f)$ are defined, contained in $\cD_\eps$, and
   $$\dist_{\cD_\eps}((\cR_n)^k(f),R_{G^{nk}(\alpha)})<\text{\rm const}\cdot\lambda^{k } \dist_{\cD_\eps}(f,R_{\alpha}).$$
   \end{proposition}
 Since a rescaling of a conjugacy is a conjugacy of a rescaling of the map, the statement would have been trivial if $\Psi_f$ and its dependence on $f$ were linear. Since they are not, the proof in \cite{GY} involves the corresponding estimates on non-linearity (cf. Lemma 4.5 from \cite{GY}). In particular, the contraction is proven in \cite{GY} for a fixed $n$ and a large enough value of the width $\eps$ (renormalization is used to increase $\eps$ if needed).

 We note that in the case of a bounded type rotation number, the proof can be significantly streamlined. We outline the argument below for the sake of completeness, the reader may skip it.
 
 \begin{proof}[Sketch of proof of Proposition~\ref{prop-conj-class}]
 Define $l=|f^{q_n}(0)|$ as before, and set $N=q_n$.  We will show that $\mathcal R_n$ contracts the nonlinearity $\max_{\Pi_{0.5\eps}} |\xi''/\xi'|$ of a linearizing chart $\xi$ if $n$ is large enough (that is, if $l$ is small enough). Distortion bound on $\Psi(lz)$ is of the order $c(\eps) \dist (f^N,R_{\alpha N})$, since the map $f^N$ is used to define the fundamental domain. The latter  distance has the order   $$\dist (f^N,R_{\alpha N}) < c(\eps) l  \max_{\Pi_{0.5\eps}} |\xi''/\xi'|.$$
This implies that $$|l \Psi''/\Psi'| \le c(\eps) l  \max_{\Pi_{0.5\eps}} |\xi''/\xi'|$$ and hence contraction of the nonlinearity of the linearizing chart for small $l$.
 \end{proof}
 
 
The infinitesimal version of this statement is \cite{GY}:
\begin{proposition}
  \label{prop-conj}
For $\alpha\in\cB_K$, $\eps>0$, there exists $n$ with the following property. If the vector field $v$ is produced by conjugation deformation:
  \begin{equation}\label{eq-conj-1}
    R_\alpha+\zeta v+o(\zeta)=(\text{Id}+\zeta h)\circ R_\alpha\circ (\text{Id}+\zeta h)^{-1},
  \end{equation}
  then $D\cR_n^k$  contracts $v$ uniformly exponentially on $\Lambda_K$.
\end{proposition}

The equation (\ref{eq-conj-1}) is equivalent to the cohomological equation
 \begin{equation}\label{eq-conj-2}
v(z)=h(z+\alpha)-h(z)\mod\ZZ.
 \end{equation}
 Evidently, any such $v$ lies in $V_0$. The converse is also true, in a very general setting, as shown via a Fourier series argument:
 \begin{proposition}
   \label{prop-conj-2}
   Suppose $\alpha\in(0,1)\setminus \QQ$ whose continued fraction convergents satisfy
   \begin{equation}\label{eq-fourier}
     \limsup \frac{\log q_{n+1}}{q_n}=0.
   \end{equation}
   Then for any  $v\in V_0$ there exists an analytic vector field $h$ in $\Pi_{0.9 \eps}$ solving the cohomological equation (\ref{eq-conj-2}).
 \end{proposition}
 Given the polynomial bound on the growth of $q_n$ in $\cB_K$, the equation (\ref{eq-fourier}) is clearly satisfied, so every zero-mean vector field comes from a conjugation deformation. This, and invariance of $\Lambda_K$, imply that $V_0$ is an invariant sub-bundle. Proposition \ref{prop-conj} implies that it is uniformly contracted, and Theorem~\ref{th-GY-2} follows.

 As a corollary, we obtain:
 \begin{theorem}
   \label{th-GY-3}
For $\eps>0, K\in \bbN$, there exists $n$ such that the hyperbolic set $\Lambda_K$ of the operator $\cR_n$ has a stable foliation by analytic submanifolds $W_\alpha\equiv W_\text{loc}^s(R_\alpha)$ of codimension one. 
   \end{theorem}

\section{Proofs of analytic linearization theorems for one-dimensional maps}

\subsection{Proof of Theorem~\ref{th-main-1}}
Fix $\alpha \in \cB_K$. In the space $\cD_{\eps/2}$, consider the submanifold $W_\alpha\equiv W_\text{loc}^s(R_\alpha)$ as in Theorem~\ref{th-GY-3}.

Assume that $f\in W_\alpha$ and 
consider the one-parameter family
$$f_w(z)\equiv f(z-w)+w \mod\ZZ.$$
We note that $f_w\in W_\alpha$ for all real $\alpha$. This is implied by Proposition~\ref{prop-conj-class} together with the following lemma:

  \begin{lemma}
    \label{lem-conj}
    There exists $s=s(K,\eps)>0$ such that for every $\alpha\in\cB_k$ any  diffeomorphism $f\colon \TT \to \TT$,  $f\in \cD_{\eps}\cap U_s(R_\alpha)$ with rotation number $\alpha$     lies in $W_\alpha$.
    \end{lemma}
  \begin{proof}

Assume $f\notin W_\alpha$. Since $\alpha$ is irrational, we have $\rho(R_\zeta \circ f) \neq \alpha$ for all real $\zeta\neq 0$. Hence $R_\zeta\circ f\notin W_\alpha$ for all real  $\zeta\neq 0$. For a complex value of $\zeta$, say $\Im \zeta >0$, the intersection $\Pi_\eps\cap \{\Im z>0\}$ is mapped strictly inside itself under $R_\zeta \circ f$ and hence the orbit of zero cannot accumulate onto itself. Thus, $R_\zeta\circ f\notin W_\alpha$ for all complex  $\zeta$.  Since $W_\alpha$ has complex codimension $1$ and is transversal to $\{R_\zeta, \zeta\in \bbC\}$,  this is not possible for $f$ in a small enough neighborhood of $R_\alpha$.
    \end{proof}

Alternatively, one could derive $f_w\in W_\alpha$ from the standard Denjoy-type estimates (see e.g. \cite{dMvS,Yoccoz2002}) which show that
an analytic map $f:\TT\to\TT$ with $\rho(f)\in\cB$ is smoothly conjugate to $R_\alpha$ on $\TT$.

\ignore{
  However, we will be needing Lemma~\ref{lem-conj} in the proof of Theorem~\ref{th-main-2}.

Fix $w\in\RR$ and let $\ell_k(z)=f_w^{q_k}(0)z$. Consider the sequence of renormalizations of $f_w$ in the sense of commuting pairs:
$$\cR_{pairs}^kf_w\equiv (\ell_k^{-1}\circ f_w^{q_k}\circ \ell_k,\ell_k^{-1}\circ f_w^{q_{k+1}}\circ \ell_k).$$
Lemma~\ref{lem-conj} implies that $\cR_{pairs}^kf_w$ converges to the space of affine interval exchange transformations exponentially fast in $k$ in the uniform metric on $\bbR$.
Since $f=f_0$ is cylinder renormalizable infinitely many times, \hl{first-return maps under $f^{q_k}, f^{q_{k+1}}$ to a neighborhood of zero are defined in a neighborhood of zero of size $const*L_k$, where $L_k = |f^{q_k}(0)|$. Since $\alpha$ has bounded type, we can expand this estimate to the whole circle: these maps are defined on a neighborhood  of size $const*L_k$ of any point $w\in \TT$. Hence }
 the maps of the pair $\cR_{pairs}^k f_w$ are defined in a uniformly large neighborhood of the origin (which, in particular, is large enough to accommodate the cylinder renormalization construction).

 Hence the exponentially fast convergence to pairs of affine maps holds in the complex neighborhood of the origin. We see that $f_w\in W_\alpha$ and Proposition~\ref{prop-pf-1} is proved.
 
 \hl{Not immediately, because of nonlinear rescalings in $\mathcal R_n$.}
}

Thus, the complex-analytic family $\{f_w, w\in \Pi_{\eps/2}\}$ intersects the complex-analytic manifold $W_\alpha\subset\cD_{\eps/2}$ over a circle $w\in\TT$. Hence the intersection contains an annulus $\Pi_\tau$.
The fact that $f_w\in W_\alpha$ for $w\in \Pi_\tau$ implies that  the orbit of $0$ under $f_w$ is infinite, stays in $\Pi_{\eps/2}$, and accumulates to $0$. Since $f_w$ is conjugate to $f$ via $z\to z+w$, we conclude that  the orbit of any point  $z\in \Pi_\tau$ under $f$ is infinite, stays in $\Pi_\eps$, and accumulates to $z$. Therefore, for small $\tau$, the map $f$ is a conformal authomorphism of an open, non-simply-connected set $\cup_{n\ge 0} f^n(\Pi_\tau)$ without fixed or periodic points. Hence this set is an annulus, and $f$ is analytically conjugate to the rotation on this annulus.

Finally, let us prove that the conjugacy extends to $\Pi_{\eps/2}$. It is sufficient to prove that $f_w\in W_{\alpha}$ for $w\in \Pi_{\eps/2}$ if $f$ is close enough to the rotation. Indeed, otherwise the family $\{f_w, w\in \Pi_{\eps/2}\}$ leaves the analytic manifold $W_\alpha$ through its relative boundary. However, since $\frac{d}{dw} f_w = -f'(z-w)+1$ is close to $0$ for $f$ close to the rotation, this cannot happen for $w\in \Pi_{\eps/2}$. This completes the proof.

\subsection{Proof of Theorem~\ref{th-main-2}}

Proposition  \ref{prop-conj-class} implies that if $f\in \cD_{\eps}$ is close to rotation and analytically conjugate to it in $\Pi_{\eps/2}$, then it belongs to the embedded analytic submanifold $W_\alpha$. We will prove the converse statement: if $f\in W_\alpha\subset \cD_{\eps}$, then it is conformally conjugate to $R_\alpha$ in $\Pi_{\eps/2}$.
  
  We are going to give two different proofs of Theorem~\ref{th-main-2}. The first one borrows from the arguments in \cite[Theorem 5.3]{GY} to derive Theorem~\ref{th-main-2} from Theorem~\ref{th-main-1}.
  The second one shares ideas with our proof of Theorem~\ref{th-main-1} and is entirely novel. 

\medskip
\noindent{\bf Version I.}
Let $f\in W_\alpha$ (and is not necessarily real-symmetric). Then the orbit $\{f^j(0)\}$ belongs to $\Pi_\eps$ and is an infinite sequence of points. So, in particular, all images of $0$ are distinct. This orbit moves holomorphically over $W_\alpha$ and hence, by the infinite-dimensional version of the $\Lambda$-Lemma \cite{BR}, so does its closure
  $$\Gamma_f\equiv \overline{\{f^j(0)\}}.$$
    Since $\Gamma_{R_\alpha}=\TT$, for every $f\in W_\alpha$, the set $\Gamma_f$ is an invariant quasi-circle. If $f$ is sufficiently close to rotation, $\Gamma_f$ is close to the real line. Note that $f$ is topologically conjugate to $R_\alpha$ on $\Gamma_f$.

    The quasicircle $\Gamma_f$ divides $\Pi_\eps$ into two cylinders. We will call $\Pi^+$ the upper and $\Pi_-$ the lower of them. Let us conformally uniformize $\Pi^+$ by
    $$h^+:\Pi^+\to \{\Im z\geq 0\}/\ZZ,\text{ so that }h^+(\Gamma_f)=\TT,\text{ and }h^+(w)=0,$$
    where $w\in \Gamma_f$.
    
    Denote $$g^+\equiv h^+\circ f\circ (h^+)^{-1}.$$
    This map preserves $\TT$ and is close to the rotation. Extend it to a neighborhood of $\TT$ using Schwarz reflection. By Theorem \ref{th-main-1}, there exists a conformal conjugacy on $\Pi_{\eps/2}$:
    $$\upsilon^+\circ g^+\circ \upsilon^{-1}=R_\alpha,\; \upsilon (0)=0.$$

    Let us carry out the same construction on $\Pi^-$ to obtain $h^-$ and $\upsilon^-$. The compositions
    $$\phi^+\equiv \upsilon^+\circ h^+,\text{ and }\phi^-\equiv \upsilon^-\circ h^-$$
    conjugate $f$ to $R_\alpha$ above and below $\Gamma_f$. They agree on $\Gamma_f$ since both of them send $w\to 0$. Since quasi-circles are holomorphically removable, these two maps glue over $\Gamma_f$ to a conformal conjugacy $\phi$ of $f$ to $R_\alpha$.

    \medskip

    \noindent
        {\bf Version II.}
        We first note:
        \begin{proposition}
          \label{prop-curve-2}
          Suppose $f\in W_\alpha$ for $\alpha\in \cB_K$. Then $f$ has an invariant topological circle $\Gamma_f\ni 0$ and $f|_{\Gamma_f}$ is topologically conjugate to $R_\alpha$.
          \end{proposition}
        \begin{proof}
          Let $V_n\ni 0$ be the fundamental rectangle in the construction of the $k$-th cylinder renormalization $\cR_n^kf$ lifted to the domain of definition of $f$. Denote
                   $$P_{nk}\equiv \left( \bigcup_{0\leq j\leq q_{{nk}+1}-1}f^j(V_{n})\right)\bigcup \left( \bigcup_{0\leq j\leq q_{nk}-1}f^j(V_{n+1})\right).$$
              This is a direct analogue of an $nk$-th dynamical partition for a homeomorphism of the circle. We set
              $$\Gamma_f\equiv \bigcap_nP_{nk}.$$
              Verifying the required properties is straightforward.
        \end{proof}
        A version of the above construction has been used previously in \cite{GaidYam1,GRY,KAM-yam}.
        
        It is not difficult to show, along the same lines, that $\Gamma_f$ is smooth, but we will not need that.

        Let $f\in W_\alpha$ and denote
        $$f_w(z)=f(z-w)+w,$$
        with $\{f_w, w\in \Pi_{\eps/2}\}\subset \mathcal D_{0.4\eps}$.
        Let $\ell_k(z)=f^{q_k}(0)z$. Consider the sequence of renormalizations of $f$ in the sense of commuting pairs:
$$\cR_{\rm pairs}^kf\equiv (\ell_k^{-1}\circ f^{q_k}\circ \ell_k,\ell_k^{-1}\circ f^{q_{k+1}}\circ \ell_k).$$
        The $\cR_{\rm pairs}^kf$ converge to the space of affine interval exchange transformations exponentially fast in $k$ in the uniform metric on a neighborhood of the circle.

        Let $w\in\Gamma_f$, then the orbit of $w$ accumulates at $0$. Replacing $f$ with its cylinder renormalization as necessary, we see that the cylinder renormalization of $f_w$ is a cylinder renormalization of  the almost affine pair $\cR_{\rm pairs}^kf$ with a different base point. Using the fact that $\rho(f|_{\Gamma_f})\in\cB$, we see that the corresponding cylinder renormalization of $f_w$ is close to a rotation.
        Hence
                $$f_w\in W_\alpha \text{ for }w\in \Gamma_f.$$

        The analytic family $f_w$ has a non-trivial intersection with the complex-analytic manifold $W_\alpha$, hence their intersection contains an open set around $0$. As a consequence, $f$ has an invariant annulus $A\supset \Gamma_f$. Arguing as in the proof of Theorem \ref{th-main-1}, we see that
                $f|_A$ is conformally conjugate to $R_\alpha$ in $\Pi_{\eps/2} $.

\section{Proof of analytic linearization in two dimensions}
\label{sec:2d}
\subsection{Extending renormalization horseshoe to 2D maps}
We recall very briefly the results of the second author from \cite{KAM-yam}. Since cylinder renormalization construction of \cite{GY} relies on conformal uniformization of a doubly connected Riemann surface, it does not directly translates to two-dimensional maps. Instead, we consider the language of renormalization of pairs. 
We set
$$T_\theta(z)=z+\theta,\text{ and }\alpha(z)=T_1(z)=z+1.$$
Let us fix a real-symmetric topological disk $W\Supset [0,1]$. We define $\cC_W$ to be the space of maps $\beta$ such that:
\begin{itemize}
\item $\beta$ is a bounded analytic map in $W$, continuous up to the boundary;
\item $\alpha\circ \beta=\beta\circ \alpha$ where defined.
\end{itemize}
Equipped with the uniform norm,  $\cC_W$ is a Banach manifold (since $\beta(z)-z$ is a $1$-periodic function).

We will refer to pairs $(\alpha,\beta)$ where $\beta\in\cC_W$ as {\it normalized commuting pairs}.
The connection with the usual definition of commuting pairs (which we will refer to as {\it non-normalized} commuting pairs, to avoid confusion) \cite{GaidYam1} is as follows. Let $\zeta=(\eta,\xi)$ be a commuting pair of analytic diffeomorphisms. The model case to consider is a map $f$ of an annulus around $\TT$ which is close to an irrational rotation $R_\theta$, and
$$\eta=f^{q_n}\text{, and }\xi=f^{q_{n+1}}.$$
Now, let $\Psi_\eta$ be a locally conformal  solution of  the functional equation
$$\Psi_\eta^{-1}\circ \eta\circ \Psi_\eta=\alpha$$
and set $\beta=\Psi_\eta^{-1}\circ\xi\circ \Psi_\eta.$ This can be accomplished following \cite{GY} in such a way that $\Psi_\eta$ analytically depends on $\eta$.

The map $\beta$ naturally projects to an analytic map $f_\beta$ of the quotient $W/_{z\sim z+1}$ and vice versa. If $\beta$ is real-symmetric and close to a translation, then $f_\beta$ is a $C^\omega$-diffeomorphism of the circle $\TT$. If its rotation number $\rho(f_\beta)\neq 0$, define 
$$\kappa(\beta)=[1/\rho(f_\beta)]\text{ and }\gamma\equiv \beta^{\kappa(\beta)}\circ \alpha.$$
Renormalization $\cR$ on the space of normalized almost commuting pairs is defined as
$$\cR:(\alpha,\beta)\mapsto (\alpha,\Psi_\beta\circ\gamma\circ\Psi_\beta^{-1}).$$
If we fix $K\in\NN$ then the set $$\hat \Lambda_K\equiv \{(\alpha,T_\theta)|\; \theta\in\cB_K\}$$
is $\cR$-invariant, and the definition of $\cR$ extends by continuity to non real-symmetric analytic maps in its neighborhood.
Renormalization $\cR$ on normalized commuting pairs is simply the cylinder renormalization defined in \cite{GY}, 
thus Theorem~\ref{th-GY-1} translates to this setting {\it mutatis mutandis}, see below.

Next, we enlarge the space of normalized commuting pairs:
\begin{definition}
We say that $(\alpha,\beta)$ is {\it a  normalized almost commuting pair} if
\begin{itemize}
\item $\beta$ is a bounded analytic map in $W$, continuous up to the boundary;
\item $[\alpha,\beta](z)\equiv \alpha\circ \beta(z)-\beta\circ \alpha(z)=o(z^2)$.
\end{itemize}
We denote the space of such maps $\beta$ by $\cAC_W$; as seen in \cite{GaidYam1} it is a Banach submanifold of the space of bounded analytic functions in $W$, continuous up to the boundary.
The definition of $\cR$ is naturally extended to these pairs; the image of an almost commuting pair is again an almost commuting pair.

Just as before, these pairs correspond to  conformal rescalings of {\it non-normalized} almost commuting pairs \cite{GaidYam1,GRY} $\zeta=(\eta,\xi)$ with 
$$[\zeta](z)=\eta\circ \xi(z)-\xi\circ\eta(z)=o(z^2).$$
  \end{definition}

As shown in \cite{KAM-yam}, the almost commutation condition does not create any new unstable directions for $\cR$ -- renormalization is
``commutativity improving''. The corresponding hyperbolicity statement is as follows \cite{KAM-yam}:

\begin{theorem}
  \label{th:hyperb2}
  Let $K\in\NN$.
There exists a domain $W$ such that the following holds. The operator $\cR$ maps an open neighborhood of $\hat\Lambda_K$ in  $\cAC_W$ to a compact subset of $\cAC_W$. Its differential at every point is a compact linear operator. The invariant set $\hat\Lambda_K$  is uniformly hyperbolic, with one-dimensional unstable direction. It has a codimension one strong stable foliation by analytic submanifolds.
  \end{theorem}

The space of almost commuting pairs is suitable for extending renormalization to two-dimensional maps. Define $\iota$ an embedding
of one-dimensional maps $h$ into the space of two-dimensional maps by
$$\iota(h):(x,y)\mapsto (h(x),h(x));$$
when applied to pairs, we use it on each of the two maps respectively.

As shown in \cite{KAM-yam}, $\cR$ can be extended as an analytic operator to the space of two-dimensional maps on a neighborhood of the set
$\iota(\hat\Lambda_K).$ The extension is somewhat technical, so we will not recall the details here. We will only note that no new unstable directions appear in the space of two-dimensional maps. The set $\iota(\hat\Lambda_K)$ remains a hyperbolic horseshoe, with a codimension-one stable foliation. This will be the starting point of our proof of Theorem~\ref{th-main-3}.

\subsection{Proof of Theorem~\ref{th-main-3}}
Let $\alpha\in\cB_K$ and let $W=W^s_\text{\rm loc}(\iota(R_\alpha)).$ Suppose, $F\in W$. Applying the same argument as in Proposition~\ref{prop-curve-2} {\it mutatis mutandis}, we see that $F$ possesses an invariant topological circle $\Gamma_F\ni(0,0)$ on which it is conjugate to $R_\alpha$. Consider the two-parameter complex family
$$F_{z,w}\equiv F(\cdot-z,\cdot-w)+(z,w).$$
Similarly to the Version II of the proof of Theorem~\ref{th-main-2}, we see that if $$(z,w)\in\Gamma_F\cap U(0),$$ then $F_{z,w}\in W$. This implies that the intersection of the two-parameter complex family $F_{z,w}$ with the codimension-one complex analytic submanifold $W$ contains a biholomorphically embedded one-dimensional complex disk $D\subset\CC^2$. The set $A=\cup F^j(D)$ is foliated by invariant circles. Since $F$ is dissipative, $A$ is topologically two real dimensional. Hence, $A$ is a conformal annulus. Montel's theorem considerations imply that $A$ is a Herman ring.

\bibliographystyle{amsalpha}
\bibliography{biblio}

\providecommand{\bysame}{\leavevmode\hbox to3em{\hrulefill}\thinspace}
\providecommand{\MR}{\relax\ifhmode\unskip\space\fi MR }
\providecommand{\MRhref}[2]{%
  \href{http://www.ams.org/mathscinet-getitem?mr=#1}{#2}
}
\providecommand{\href}[2]{#2}
\begin{thebibliography}{dMvS93}

\bibitem[Arn61]{Arnold61}
V.I. Arnol'd, \emph{Small denominators. {I}. {M}apping the circle onto itself},
  Izv. {A}kad. {N}auk {S}{S}{S}{R} Ser. {M}at. \textbf{25} (1961), 21--86.

\bibitem[BR]{BR}
L.~Bers and H.~L. Royden, \emph{Holomorphic families of injections}, Acta
  {M}ath. \textbf{157}, 259--286.

\bibitem[dMvS93]{dMvS}
W.~de~Melo and S.~van Strien, \emph{One-dimensional dynamics}, Springer, 1993.

\bibitem[GRY21]{GRY}
D.~Gaidashev, R.~Radu, and M.~Yampolsky, \emph{Renormalization and {S}iegel
  disks for complex {H}\'e non maps}, J. {E}ur. {M}ath. {S}oc. \textbf{23}
  (2021), 1053–1073.

\bibitem[GY20]{GaidYam1}
D.~Gaidashev and M.~Yampolsky, \emph{Renormalization of almost commuting
  pairs}, Invent. math. \textbf{221} (2020), 203–236.

\bibitem[GY22]{GY}
N.~Goncharuk and M.~Yampolsky, \emph{Analytic linearization of conformal maps
  of the annulus}, Advances in Mathematics \textbf{409} (2022).

\bibitem[KS87]{KhSin87}
K.~Khanin and Y.~Sinai, \emph{A new proof of {M}. {H}erman's theorem.},
  Commun.Math. Phys. \textbf{112} (1987), 89–101.

\bibitem[Ris99]{Risler}
E.~Risler, \emph{Lin\'{e}arisation des perturbations holomorphes des rotations
  et applications}, M\'{e}m. Soc. Math. Fr. (N.S.) (1999), no.~77.

\bibitem[Yam21]{KAM-yam}
M.~Yampolsky, \emph{{K}{A}{M}-renormalization and {H}erman rings for 2{D}
  maps}, {C}. {R}. {M}ath. {R}ep. {A}cad. {S}ci. {C}anada \textbf{43} (2021),
  no.~2, 78--86.

\bibitem[Yoc02]{Yoccoz2002}
J.-C. Yoccoz, \emph{Analytic linearization of circle diffeomorphisms},
  Dynamical systems and small divisors ({C}etraro, 1998), Lecture Notes in
  Math., vol. 1784, Springer, Berlin, 2002, pp.~125--173.

\end{thebibliography}

 \end{document}